\documentclass[12pt]{amsart}
\usepackage{times}
\usepackage{amsfonts}
\usepackage{amssymb}
\usepackage{bbm}
\usepackage{times}
\usepackage{amssymb}
\usepackage{amscd}
\usepackage{graphicx}

\usepackage{amsmath}
\usepackage{amssymb}

\usepackage{amsmath}
\usepackage{amsfonts}
\usepackage{amscd}
\usepackage{latexsym}
\usepackage{amsthm}
\usepackage{amssymb}
\usepackage{amsmath}
\usepackage{setspace}
\theoremstyle{plain}
\newtheorem{theorem}{Theorem}[section]
\newtheorem{corollary}[theorem]{Corollary}

\newtheorem{definition}[theorem]{Definition}

\begin{document}

\vskip 0.5cm

\title[generalized tracially approximated ${\rm C^*}$-algebras] {certain properties of weakly tracially approximated ${\rm \textbf{C}^*}$-algebras}
\author{Qingzhai Fan and Jiahui Wang}

\address{Qingzhai Fan\\ Department of Mathematics\\  Shanghai Maritime University\\
Shanghai\\China
\\  201306 }

\address{Jiahui  Wang\\ Department of Mathematics\\ Shanghai Maritime University\\
Shanghai\\China
\\  201306 }
\email{641183019@qq.com}

\thanks{2000 \emph{Mathematics Subject Classification\rm{:}} 46L35, 46L05, 46L80}

\begin{abstract} Let  $\Omega$ be a class of unital ${\rm C^*}$-algebras. Let  ${\rm WTA}\Omega$ denote the class
 of ${\rm C^*}$-algebras which can be weakly  tracially approximated by ${\rm C^*}$-algebras in $\Omega$. We show that the following properties of unital ${\rm C^*}$-algebra in a class of $\Omega$ are preserved by unital simple ${\rm C^*}$-algebra in the class of $\rm WTA\Omega$:
 $(1)$ uniform property $\Gamma$, $(2)$  a type of tracial nuclear dimension at most $n$,  $(3)$  weakly  $(m, n)$-divisible.
\end{abstract}

\maketitle

\section{Introduction}

A crucial step in the classification of nuclear stably finite unital separable ${\rm C^*}$-algebras (the classification of of nuclear
${\rm C^*}$-algebras are called Elliott program, which have begun with the ${\rm K}$-theoretical
classification of AF algebras in \cite{E1})
 was Lin's  abstract tracial approximation structure  which inspired by Elliott-Gong's decomposition theorem  (\cite{E6}) and Gong's decomposition theorem (\cite{G1}) for simple AH algebras.
 Elliott and  Niu first  studied this tracial approximation structure in \cite{EZ} and they showed that some class of ${\rm C^*}$-algebras  and $K$-theory properties of
 ${\rm C^*}$-algebras in a class $\Omega$ are inherited by simple unital ${\rm C^*}$-algebras in the
 class  of tracial approximated by ${\rm C^*}$-algebras in the class $\Omega$.
 The properties of a class of ${\rm C^*}$-algebra $\Omega$ preserved in simple unital ${\rm C^*}$-algebras in the
 class  of tracial approximated by ${\rm C^*}$-algebras in the class $\Omega$,  can be used to study the crossed   product   ${\rm C^*}$-algebras,  which obtained by finite group  action
  on ${\rm C^*}$-algebra which has  tracial Rokhlin property (\cite{FF1}).

Centrally  large subalgebra was  introduced  \cite{AN} by Phillips and  Archey   as  abstractions of Putnam's orbit breaking subalgebra of the crossed product algebra ${\rm C}^*(X,\mathbb{Z},\sigma)$ of the Cantor set by a minimal homeomorphism in \cite{P}.
Whether the properties of centrally  large subalgebra can be preserved in the original algebra has
important applications in studying the properties of certain crossed  product ${\rm C^*}$-algebras   which obtained by integer group $\mathbb{Z}$  action
  on compact topological spaces $X$ (\cite{EN2}).

Inspired  by   centrally  large subalgebra and tracial approximation structure, Elliott, Fan and Fang introduced a class of  unital weakly   tracial approximation ${\rm C^*}$-algebras in \cite{EFF}.  The notion  generalizes  both Archey and Phillips's centrally large subalgebras and Lin's notion of tracial approximation.


Let $\Omega$ be a class of finite dimension ${\rm C^*}$-algebras. Let $A$ be an infinite-dimensional unital simple $\rm C^{*}$-algebra and let
$B\subseteq A$ be a centrally large
    	subalgebra of $A$ such that $B$ has  tracial topological rank zero. In \cite{FFW1}, Fan, Fang and Wang show that $A$ is in the  class
 of ${\rm C^*}$-algebras which can be weakly  tracially approximated by ${\rm C^*}$-algebras in $\Omega$

In  \cite{EFF},  Elliott,  Fan, and Fang  show that  the following  properties  of unital
${\rm C^*}$-algebras in a class $\Omega$ are inherited by unital simple  ${\rm C^*}$-algebras in the class of ${\rm WTA}\Omega$:
$(1)$ tracially $\mathcal{Z}$-absorbing,
$(2)$ tracial nuclear dimension at most $n$,
$(3)$ the property $\rm SP$, and
$(4)$ $m$-almost divisible.

In this paper, we shall prove the following  results:

  Let $\Omega$ be a class of   unital
${\rm C^*}$-algebras which  have  uniform property $\Gamma$ (see Definition \ref{def:2.9}).  Then  $A$  has  uniform property $\Gamma$ for  any  simple  unital stably finite  ${\rm C^*}$-algebra $A\in{\rm  WTA}\Omega$ (Theorem  \ref{thm:3.1}).
\vskip 0.1cm
Let $\Omega$ be a class of   unital nuclear
${\rm C^*}$-algebras  which have a  type of  tracial nuclear dimension at most $n$ (see Definition \ref{def:2.8}).   Then $A$ has the  type of tracial nuclear dimension at most $n$ for  any  simple  unital   ${\rm C^*}$-algebra $A\in{\rm  WTA}\Omega$ (Theorem  \ref{thm:3.4}).
\vskip 0.1cm
Let $\Omega$ be a class of unital ${\rm C^*}$-algebras which are weakly $(m, n)$-divisible (see Definition \ref{def:2.2}). Let $A\in {\rm  WTA}\Omega$
    	be a simple unital stably finite $C^{*}$-algebra such that for any integer $n\in \mathbb{N}$ the ${\rm C^*}$-algebra ${\rm M}_{n}(A)$
    	belongs to the class ${\rm WTA}\Omega$. Then $A$ is secondly weakly $(m, n)$-divisible (see Definition \ref{def:2.3}) (Theorem  \ref{thm:3.7}).

As applications, the following known results follow from these results.

Let $A$ be an infinite-dimensional stably finite unital simple separable  ${\rm C^*}$-algebra. Let $B\subseteq A$ be a centrally large subalgebra in $A$ such that $B$ has uniform property $\Gamma$. Then   $A$ has uniform property $\Gamma$. This result was obtained by Fan and Zhang in \cite{FZ}

Let $\Omega$ be a class of  stably finite unital
${\rm C^*}$-algebras such that for  any $B\in \Omega$,  $B$ has  uniform property $\Gamma$. Then $A$ has  uniform property $\Gamma$ for  any  simple  unital ${\rm C^*}$-algebra $A\in \rm{TA}\Omega$.
This result was obtained  by Fan and Zhang in \cite{FZ}.

Let $A$ be a unital simple stably finite  separable  ${\rm C^*}$-algebra. Let $B\subseteq A$ be a centrally large subalgebra of  $A$ such that $B$ is weakly $(m, n)$-divisible. Then   $A$ is  secondly weakly $(m, n)$-divisible. This result was obtained  by Fan, Fang, and Zhao in \cite{FFZ}.

\section{Preliminaries and definitions}
Let $A$ be a ${\rm C^*}$-algebra. For two positive elements $a,b\in A$ we say that $a$ is that $a$ is Cuntz subequivalent to $b$  and write  $a\precsim b$
if there is a sequence $(r_n)_{n=1}^\infty$
of elements of $A$ such that $$\lim_{n\to \infty}\|r_nbr_n^*-a\|=0.$$
 We know that Cuntz equivalent is an equivalence relation. We write $a\sim b$ and say $a$ and $b$ are Cuntz equivalent if $a\precsim b$ and $b\precsim a$.

Let $A$ be a ${\rm C^*}$-algebra and let ${\rm M}_{\infty}(A)=\bigcup_{n\in \mathbb{N}}{\rm M}_n(A)$ where ${\rm M}_n(A)$ is included into  ${\rm M}_{n+1}(A)$ by copy ${\rm M}_n$  into the top left corner of ${\rm M}_{n+1}(A)$.
Suppose $a, b\in {\rm M}_{\infty}(A)_+$. Then $a\in {\rm M}_n(A)$ and $b\in {\rm M}_m(A)$ for some $n,m\in \mathbb{N}$.
e say that $a$ is that $a$ is Cuntz subequivalent to $b$  and write  $a\precsim b$ if $a\oplus 0_{\max\{n-m,0\}}\precsim a\oplus 0_{\max\{m-n,0\}}$ as elements in
 ${\rm M}_{\max\{n,m\}}(A)$ where $0_n$ is the zero element of ${\rm M}_n(A)$.



For a positive element $a$ in a ${\rm C^*}$-algebra $A$ and $\varepsilon>0$,
 we denote by $(a-\varepsilon)_+$ the positive element obtained by applying  the functional calculus to the function
  $f(t)={\max (0, t-\varepsilon)}$ where $t\in [0,\|a\|]$.


The following facts are  well known ({\rm \cite{PPT}, \cite{HO}, \cite{P3}, \cite{RW}.}).
\begin{theorem} \label{thm:2.1} Let $A$ be a ${\rm C^*}$-algebra.

 $(1)$ Let $a, b\in A_+$ and any  $\varepsilon>0$  be such that
$\|a-b\|<\varepsilon$.  Then   $(a-\varepsilon)_+\precsim b$.

$(2)$ Let $a, p$ be positive elements in $A$ with $p$ a projection.
 If $p\precsim a,$ and $p$ is  not equivalent to $a$,  then there is a nonzero element $b$ in $A$ such that  $bp=0$ and $b+p\precsim a$.

 $(3)$ Let $a$ be a  positive element  of   ${\rm C^*}$-algebra $A$ and $a$ is
not Cuntz equivalent to a projection.  Then there exists non-zero positive element $d\in A$ such that
 $(a-\delta)_++d\precsim a.$

\end{theorem}

The property  of weakly $(m, n)$-divisible was introduced by  Robert and  R{\o}rdam in \cite{KM}.
\begin{definition}{\rm (\cite{KM}.)}\label{def:2.2}
    	Let $A$ be a unital ${\rm C^*}$-algebra. Let $m, n\geq 1$ be two integers. $A$ is said to be weakly $(m, n)$-divisible,
    	if for every $a\in {\rm M}_{\infty}(A)_{+}$ and any $\varepsilon>0$, there exist elements
    	$x_{1}, x_{2}, \cdots, x_{n}\in {\rm M}_{\infty}(A)_{+}$, such that
    	$m\left\langle x_{j}\right\rangle \leq \left\langle a \right\rangle$, for all $j=1, 2, \cdots, n$ and
    	$\left\langle (a-\varepsilon)_{+} \right\rangle\leq\left\langle x_{1} \right\rangle+\left\langle x_{2} \right\rangle+\cdots+\left\langle x_{n} \right\rangle$.
    \end{definition}
    \begin{definition}{\rm (\cite{FFZ}.)}\label{def:2.3}
    	Let $A$ be a unital ${\rm C^*}$-algebra. Let $m, n\geq 1$ be two integers. $A$ is said to be secondly weakly $(m, n)$-divisible,
    	if for every $a\in {\rm  M}_{\infty}(A)_{+}$ and any $\varepsilon>0$, there exist elements
    	$x_{1},  x_{2}, \cdots,  x_{n}\in {\rm M}_{\infty}(A)_{+}$ such that
    	$m\left\langle x_{j}\right\rangle \leq \left\langle a \right\rangle +\left\langle a \right\rangle$
    	for all $j=1, 2, \cdots, n$, and $\left\langle (a-\varepsilon)_{+} \right\rangle\leq\left\langle x_{1} \right\rangle+\left\langle x_{2} \right\rangle+\cdots+\left\langle x_{n} \right\rangle$.
    \end{definition}

Let  $\Omega$ be a class of unital ${\rm C^*}$-algebras. Elliott, Fan,  and Fang defined as follows  the class
 of ${\rm C^*}$-algebras which can be weakly  tracially approximated by ${\rm C^*}$-algebras in $\Omega$, and denoted this class by ${\rm WTA}\Omega$ in \cite{EFF}.

\begin{definition}\label{def:2.6}{\rm (\cite{EFF}.)}  A  simple unital ${\rm C^*}$-algebra $A$ is   said to belong to the class ${\rm WTA}\Omega$  if, for any
 $\varepsilon>0$, any finite
subset $F\subseteq A$, and any  non-zero element $a\geq 0$, there
exist a  projection $p\in A$, an  element $g\in A$ with $0\leq g\leq 1$,
  and a unital ${\rm C^*}$-subalgebra $B$ of $A$ with
$g\in B, ~1_B=p$, and $B\in \Omega$, such that

$(1)$  $(p-g)x\in_{\varepsilon} B, ~ x(p-g)\in_{\varepsilon} B$ for all $x\in  F$,

$(2)$ $\|(p-g)x-x(p-g)\|<\varepsilon$ for all $x\in F$,

$(3)$ $1-(p-g)\precsim a$, and

$(4)$ $\|(p-g)a(p-g)\|\geq \|a\|-\varepsilon$.
\end{definition}



Winter and Zacharias  introduced the notion of  nuclear dimension for  ${\rm C^*}$-algebras in \cite{WW3}.
\begin{definition}{\rm (\cite{WW3}.)}\label{def:2.7} A ${\rm C^*}$-algebra $A$ has nuclear dimension at most $n$, and denoted by ${\rm dim_{nuc}}\leq n$, if there exists a
net $(F_\lambda,\psi_\lambda, \varphi_\lambda)_{\lambda\in \Lambda}$ such that the $F_\lambda$ are finite-dimensional ${\rm C^*}$-algebras, and such that $\psi_\lambda:A\to F_\lambda$ and $\varphi_\lambda:F_\lambda\to A$ are completely positive maps satisfying

$(1)$ $\psi_\lambda\varphi_\lambda(a)\to a$ uniformly on finite subsets of $A$,

$(2)$ $\|\psi_\lambda\|\leq 1$,

$(3)$ for each $\lambda$,
$F_\lambda$ decomposes into $n+1$ ideals
$F_\lambda={F_\lambda}^0+\cdots+{F_\lambda}^n$
such that
$\varphi_\lambda|_{{F_\lambda}^i}$ is
completely positive contractive order zero maps (which means preserves orthogonality, i.e., $\varphi_\lambda(e)\varphi_\lambda(f)=0$ for all $e,~ f\in {{F_\lambda}^i}$ with $ef=0$) for $i=0,~\cdots,~n$.
\end{definition}

Inspired by  Hirshberg and  Orovitz's  tracial $\mathcal{Z}$-absorption in \cite{HO},
 Fu introduced  some notion of  tracial nuclear dimension in his doctoral dissertation  \cite{FU} (see also  \cite{FL}).

 \begin{definition}{\rm (\rm \cite{FU}.)}\label{def:2.8}
		Let $A$ be a  $\rm C^{*}$-algebra. Let $n\in \mathbb{N}$ be an integer.
 $A$ is  said to have  second  type of  tracial nuclear dimension
		at most $n$,  and denoted by ${\rm {T^2dim_{nuc}}}(A)\leq n$, if for any finite positive subset $\mathcal{F}\subseteq A$, for any $\varepsilon>0$ and for any nonzero positive element $a\in A$, there exist a finite dimensional $\rm C^{*}$-algebra $F=F_{0}\oplus\cdots\oplus F_{n}$ and  completely positive  maps $\psi:A\rightarrow F$, $\varphi:F\rightarrow A$ such that
		
		$(1)$ for any  $x\in F$,  there exists $x'\in A_{+}$, such that $x'\precsim a$ and $\|x-x'-\varphi\psi(x)\|<\varepsilon$,
		
		$(2)$  $\|\psi\|\leq1$, and
		
		$(3)$ $\varphi|_{F_{i}}$ is  a contractive completely positive  order zero map for $i=1, \cdots, n$.
		
	\end{definition}

 Inspired by Fu's second type of tracial nuclear dimension  in  \cite{FU}, we shall introduce a new  type  of  tracial nuclear dimension for unital $\rm C^{*}$-algebra.

	\begin{definition}\label{def:2.9}
		Let $A$ be a unital $\rm C^{*}$-algebra. Let $n\in \mathbb{N}$ be an integer.
 $A$ will said to have a new  type  of  tracial nuclear dimension
		at most $n$, if for any finite positive subset $\mathcal{F}\subseteq A$, for any $\varepsilon>0$ and for any nonzero positive element $a\in A_{+}$, there exist a finite dimensional $\rm C^{*}$-algebra $F=F_{0}\oplus\cdots\oplus F_{n}$ and  completely positive  maps $\psi:A\rightarrow F$, $\varphi:F\rightarrow A$ such that
		
		$(1)$ for any  $x\in F$,  there exists $x'\in A_{+}$, such that $\|x-x'-\varphi\psi(x)\|<\varepsilon$,
		
		$(2)$ $(1_{A}-\varphi\psi(1_{A})-\varepsilon)_+\precsim a$,
		
		$(3)$ $\|\psi\|\leq1$, and
		
		$(4)$ $\varphi|_{F_{i}}$ is  a contractive completely positive  order zero map for $i=1, \cdots, n$.
		
	\end{definition}

Let $A$ be a unital $\rm C^{*}$-algebra. It is easy to know that
 ${\rm {T^2dim_{nuc}}}(A)\leq n$ implies   $A$ has the  new  type of  tracial nuclear dimension at most $n$.

 Uniform property $\Gamma$ was introduced by  J. Castillejos, S. Evington, A. Tikuisis, S. White,  and W. Winter, which was used to prove that $\mathcal{Z}$-stable imply that  finite nuclear dimension in  \cite{CETWW}.




We recall the equivalent local refinement of Property $\Gamma$ from  Proposition 2.4 of \cite{CETWW1}.

\begin{theorem} \label{thm:2.11} {\rm (\cite{CETWW1}.)} Let $A$ be a separable ${\rm C^*}$-algebra
with $T(A)$ (the trace state space of $A$) nonempty and compact. Then the following are equivalent:

 $(1)$ $A$ has uniform property $\Gamma$.

 $(2)$ For any finite subset $F\subseteq A$, any $\varepsilon>0$, and any integer $n\in \mathbb{N}$, there exist pairwise orthogonal positive contractions $e_1, \cdots, e_n\in A$ such that for $i=1, \cdots, n$, and $a\in F$, we have
 $\|e_ia-ae_i\|<\varepsilon$ and $$\sup_{\tau\in T(A)}\|\tau(ae_i)-\frac{1}{n}\tau(a)\|<\varepsilon.$$
\end{theorem}

\section{The main results}

\begin{theorem}\label{thm:3.1}
     	Let $\Omega$ be a class of unital separable  $\rm C^{*}$-algebra such that $T(B)$ (the trace state space of $B$) is nonempty and compact and $B$ has  uniform property $\Gamma$ for any $B\in \Omega$.
     	Then $A$ has uniform property $\Gamma$ for any simple infinite-dimensional separable unital stably finite $\rm C^{*}$-algebra $A\in {\rm WTA}\Omega$.
     \end{theorem}

\begin{proof} Since $A$ is a stably finite $\rm C^{*}$-algebra, then  $T(A)$ is nonempty.  This together with the unitality of $A$ implies that $T(A)$ is compact.
  By Theorem \ref{thm:2.11}: (2),  we need to show that for fixed  finite subset $F=\{a_1, a_2, \cdots, a_k\}$ of $A$ (we may assume that $\|a_j\|\leq 1$ for all $j=1, \cdots,  k$),   any $\varepsilon>0$, and  any integer $n\in \mathbb{N}$, there exist pairwise orthogonal positive contractions $e_1, \cdots,  e_n\in A$ such that
 $\|e_ia_j-a_je_i\|<\varepsilon$ and $$\sup_{\tau\in T(A)}\|\tau(a_je_i)-\frac{1}{n}\tau(a_j)\|<\varepsilon,$$
 for $i=1, \cdots,  n,$ and  $j=1, \cdots, k$.

 For  $\varepsilon>0$, we choose $0<\delta<\varepsilon$  with $X=[0,1],  f(t)=t^{1/2},  g(t)=(1-t)^{1/2}$
  according to  Lemma 2.5.11 in \cite {L2}.

 Since $A$ is an infinite-dimensional  unital simple separable ${\rm C^*}$-algebra, by Corollary 2.5 in \cite{P3},
 there exists a nonzero positive element $a\in A$  such that $\delta>d_\tau(a)
 =\lim_{n\to\infty}\tau(a^{1/n})$ for any $\tau\in T(A)$.

 For $F=\{a_1, a_2, \cdots, a_k\}$, any $\delta>0$, and nonzero $a\in A_+$,
     	 since $A\in {\rm  WTA}\Omega$, there exist a projection $p\in A$, an element
     	 $g\in A$ with $\|g\|\leq1$ and a unital $\rm C^{*}$-subalgebra $B$ of $A$
     	with $g\in B, 1_{B}=p$ and $B\in\Omega$ such that
     	
     	$(1)$ $(p-g)x\in_{\delta}B, ~x(p-g)\in_{\delta}B$ for any $x\in F$,
     	
     	$(2)$ $\|(p-g)x-x(p-g)\|<\delta$ for any $x\in F$, and
     	
     	$(3)$ $1_{A}-(p-g)\precsim a$.

By $(2)$, we have

$(1)'$ $\|(1_{A}-(p-g))a_j-a_j(1_{A}-(p-g))\|<\delta$  for any $j=1, \cdots, k$.

By  $(3)$ and by Proposition 1.19 of \cite{P3}, we have $$d_\tau(1_{A}-(p-g))\leq d_\tau(a)$$ for any $\tau\in T(A)$.  Since
$d_\tau(a)<\delta$, and $\tau(1_{A}-(p-g))\leq d_\tau(1_{A}-(p-g))$, we have
  $$\tau(1_{A}-(p-g))<\delta$$ for any $\tau\in T(A)$.

 By the choice of  $\delta$, by  $(2)$, $(1)'$ and  by Lemma 2.5.11 in \cite {L2}, one has

$\|(1_{A}-(p-g))^{1/2}a_j-a_j(1_{A}-(p-g))^{1/2}\|<\varepsilon,$ and

$\|(p-g)^{1/2}a_j-a_j(p-g)^{1/2}\|<\varepsilon.$

By $(1)$, there exists   $a_j'\in B$  such that

$(2)'$ $\|(p-g)a_j-a_j'\|<\delta.$

By $(1)'$ and $(2)'$, one has

  $\|a_j-a_j'-(1_{A}-(p-g))^{1/2}a_j(1_{A}-(p-g))^{1/2}\|$

  $ =\|(1_{A}-(p-g))a_j+(p-g)a_j-a_j'-
  (1_{A}-(p-g))^{1/2}a_j(1_{A}-(p-g))^{1/2}\|$

  $\leq\|(1_{A}-(p-g))a_j-(1_{A}-(p-g))^{1/2}a_j(1_{A}-(p-g))^{1/2}\|+
  \|(p-g)a_j-a_j'\|$$<\varepsilon+\delta<2\varepsilon,$
   and

   $\|(p-g)^{1/2}a_j(p-g)^{1/2}-a_j'\|\leq\|(p-g)^{1/2}a_j(p-g)^{1/2}-
   (p-g)a_j\|$$+\|(p-g)a_j-a_j'\|$
   $<\varepsilon+\delta<2\varepsilon$
     for all $1\leq j\leq k$.

 For $\varepsilon>0$, and any integer $n\in \mathbb{N}$, we choose $\delta''=\delta''(\varepsilon,n)$ (with $\delta''<\varepsilon$) sufficiently small  such that satisfying  Lemma 2.5.12 in \cite {L2}.

 For $\delta''/2>0$, finite subset $\{a_1', a_2', \cdots,  a_k',  (p-g), (p-g)^{1/2}\}$ of $B$ and $n\in \mathbb{N}$, since $B$ has uniform property $\Gamma$, there exist pairwise orthogonal
 positive contractions $e_1', \cdots, e_n'\in B$ such that for $i=1, \cdots, n,  j=1, \cdots, k$, one has
 $$\|e_i'a_j'-a_j'e_i'\|<\delta''/2,  \|e_i'(p-g)-(p-g)e_i'\|<\delta''/2,$$$$\|e_i'(p-g)^{1/2}-(p-g)^{1/2}e_i'\|<\delta''/2,$$ and $$\sup_{\tau\in T(B)}\|\tau(a_j'e_i')-\frac{1}{n}\tau(a_j')\|<\delta''/2.$$

Since $\|(p-g)^{1/2}e_i'-e_i'(p-g)^{1/2}\|<\delta''$, and $e_i'e_j'=0, i\neq j$,   so we have

$\|(p-g)^{1/2}e_i'(p-g)^{1/2}(p-g)^{1/2}e_j'(p-g)^{1/2}\|$

$\leq\|(p-g)^{1/2}e_i'(p-g)^{1/2}(p-g)^{1/2}e_j'(p-g)^{1/2}-
(p-g)e_i'(p-g)^{1/2}e_j'(p-g)^{1/2}\| $
$+\|(p-g)e_i'(p-g)^{1/2}e_j'(p-g)^{1/2}-
(p-g)e_i'e_j'(p-g)\|$

$<\delta''/2+\delta''/2=\delta''$.

By the choice of $\delta$, since $(p-g)^{1/2}e_i'(p-g)^{1/2}$ is a contraction, by the proof the  Lemma 2.5.12 in \cite{L2}, one can fine pairwise orthogonal
 positive contractions  $e_i$ $(i=1,\cdots, n)$,  such that $$\|(p-g)^{1/2}e_i'(p-g)^{1/2}-e_i\|<\varepsilon.$$

We have

$\|a_je_i-a_j'e_i'\|$$\leq\| a_je_i-a_j(p-g)^{1/2}e_i'(p-g)
^{1/2}\|$

$+\|a_j(p-g)^{1/2}e_i'(p-g)^{1/2}-a_j(p-g)^{1/2}(p-g)^{1/2}e_i'\|$

$+\|a_j(p-g)^{1/2}(p-g)^{1/2}e_i'-(p-g)^{1/2}a_j(p-g)^{1/2}e_i'\|$

$+\|(p-g)^{1/2}a_j(p-g)^{1/2}e_i'-a_j'e_i'\|\leq 4\varepsilon$.

With the same argument, one has
$$ \|e_ia_j-e_i'a_j'\|<4\varepsilon$$ for $i=1, \cdots, n, j=1, \cdots, k$.

Since $\|e_i'a_j'-a_j'e_i'\|<\delta''/2$, so one has

$\|a_je_i-e_ia_j\| $

$\leq\|a_je_i-a_j'e_i'\|+\|a_j'e_i'-e_i'a_j'\|+\|e_i'a_j'-e_ia_j\|
$

$<4\varepsilon+4\varepsilon+\delta''/2<9\varepsilon,$ for $i=1,\cdots, n, j=1, \cdots, k$.

Since $\|a_je_i-a_j'e_i'\|<4\varepsilon$ for any $\tau\in T(A)$, one has
$$|\tau(a_je_i)-\tau(a_j'e_i') |<4\varepsilon.$$

Since $\|a_i-a_i'-(1_{A}-(p-g))^{1/2}a_j(1_{A}-(p-g))^{1/2}\|<2\varepsilon,$ we have
$$|\tau(a_j)-\tau(a_j')-\tau((1_{A}-(p-g))^{1/2}a_j(1_{A}-(p-g))^{1/2})|<2\varepsilon.$$

Therefore, we have

$|\tau(a_je_i)-\frac{1}{n}\tau(a_j)|$

$=|\tau(a_je_i)-\tau(a_j'e_i')|+|\tau(a_j'e_i')-\frac{1}{n}\tau(a_j')|
+|(\frac{1}{n}\tau(a_j')-\frac{1}{n}\tau(a_j)|$

$\leq4\varepsilon+|\tau(a_j'e_i')-\frac{1}{n}\tau(a_j')|+\frac{1}{n}
(2\varepsilon+\tau((1-(p-g))^{1/2}a_j(1-(p-g))^{1/2}))
$

$\leq4\varepsilon+|\tau(a_j'e_i')-\frac{1}{n}\tau(a_j')|+\frac{1}{n}
(2\varepsilon+\tau(1-(p-g))$

$\leq5\varepsilon+|\tau(a_j'e_i')-\frac{1}{n}\tau(a_j')|$.

Therefore, one has

$\sup_{\tau\in T(A)}|\tau(a_je_i)-\frac{1}{n}\tau(a_j)|$

$\leq \sup_{\tau\in T(A)}|\tau(a_j'e_i')-\frac{1}{n}\tau(a_j')|+5\varepsilon$

$\leq \frac{1}{\tau(p)}\sup_{\tau\in T(B)}|\tau(a_j'e_i')-\frac{1}{n}\tau(a_j')|+5\varepsilon$

$\leq \frac{1}{1-\delta}\sup_{\tau\in T(B)}|\tau(a_j'e_i')-\frac{1}{n}\tau(a_j')|+5\varepsilon$

$<5\varepsilon +\frac{\delta''}{1-\delta}$
$<6\varepsilon.$

By Theorem 2.11:(2), $A$ has uniform property $\Gamma$.
\end{proof}

The following two corollaries were obtained by Fan and
Zhang in \cite{FZ}.

\begin{corollary}(\cite{FZ}) \label{cor:3.2}Let $A$ be an infinite-dimensional stably finite unital simple  ${\rm C^*}$-algebra. Let $B\subseteq A$ be a centrally large subalgebra of $A$ such that $B$ has uniform property $\Gamma$. Then   $A$ has uniform property $\Gamma$.\end{corollary}

\begin{corollary}(\cite{FZ}) \label{cor:3.3} Let $\Omega$ be a class of  stably finite unital
${\rm C^*}$-algebras such that for  any $B\in \Omega$,  $B$ has  uniform property $\Gamma$. Then $A$ has  uniform property $\Gamma$ for  any  simple  unital ${\rm C^*}$-algebra $A\in \rm{TA}\Omega$.\end{corollary}

    \begin{theorem} \label{thm:3.4}
    	Let $\Omega$ be a class of   unital nuclear
${\rm C^*}$-algebras  which have the new type of  tracial nuclear dimension at most $n$ (in the sense of Definition  \ref{def:2.9}).   Then $A$ has the  new type of tracial nuclear dimension at most $n$  for  any  simple  unital  ${\rm C^*}$-algebra $A\in{\rm  WTA}\Omega$.
    \end{theorem}

    \begin{proof}
    	We must show that for any finite positive  subset $\mathcal{F}=\{a_1, a_2,$ $\cdots, a_k\}$ $\subseteq A$, for any $\varepsilon>0$, and for any nonzero positive  element $b\in A$, there exist a finite dimensional $\rm C^{*}$-algebra $F=F_{0}\oplus\cdots\oplus F_{n}$ and  completely positive maps $\psi:A\rightarrow F$,~ $\varphi:F\rightarrow A$ such that
    	
    	$(1)$ for any $x\in F$, there exists $\overline{x}\in A_+$, such that $\|x-\overline{x}-\varphi\psi(x)\|<\varepsilon$,
    	
    	$(2)$ $(1_{A}-\varphi\psi(1_{A})-\varepsilon)_+\precsim b$,
    	
    	$(3)$ $\|\psi\|\leq1$, and
    	
    	$(4)$ $\varphi|_{F_{i}}$ is a completely positive  contractive order zero map  for $i=1, \cdots, n$.
    	
    	By Lemma 2.3 of \cite{L2}, there exist positive elements $b_{1}, b_{2}\in A$ of norm one such that
    	$b_{1}b_{2}=0, b_{1}\sim b_{2}$ and $b_{1}+b_{2}\precsim b$.
    	
    	Given $\varepsilon'>0$, for $H=\mathcal{F}\cup\{b\}$, since $A\in {\rm WTA}\Omega$, there exist a   projection $p\in A$, an element $g\in A$ with $\|g\|\leq1$ and a unital $\rm C^{*}$-subalgebra $B$ of $A$ with $g\in B, 1_{B}=p$ and $B\in\Omega$ such that
    	
    	$(1)'$ $(p-g)x\in_{\varepsilon'}B, x(p-g)\in_{\varepsilon'}B$ for any $x\in H$,
    	
    	$(2)'$ $\|(p-g)x-x(p-g)\|<\varepsilon'$ for any $x\in H$,
    	
    	$(3)'$ $1_{A}-(p-g)\precsim b_{1}\sim b_{2}$, and
    	
    	$(4)'$ $\|(p-g)b(p-g)\|\geq 1-\varepsilon'$.
    	
    	By $(2)'$ and Lemma 2.5.11 of \cite{L2}, with sufficiently small $\varepsilon'$, we can get
    	
    	$(5)'$ $\|(p-g)^{\frac{1}{2}}x-x(p-g)^{\frac{1}{2}}\|<\varepsilon$ for any $x\in H$, and
    	
    	$(6)'$ $\|(1_{A}-(p-g))^{\frac{1}{2}}x-x(1_{A}-(p-g))^{\frac{1}{2}}\|<\varepsilon$ for any $x\in H$.
    	
    	By $(1)'$ and $(5)'$, with sufficiently small $\varepsilon'$, there exist positive  elements $a'_{1}, \cdots, a'_{n}\in B$
    	and a positive element $b_{2}'\in B$ such that

    $\|(p-g)^{\frac{1}{2}}a_{i}(p-g)^{\frac{1}{2}}-a'_{i}\|<\varepsilon$ for $1\leq i\leq k$, and
    $\|(p-g)^{\frac{1}{2}}b_{2}(p-g)^{\frac{1}{2}}-b_{2}'\|<\varepsilon$.
    	
    	Therefore, one has
    	
    	 $\|a_{i}-a'_{i}-(1_{A}-(p-g))^{\frac{1}{2}}a_{i}(1_{A}-(p-g))^{\frac{1}{2}}\|$
    	
    	$\leq\|a_{i}-(p-g)a_{i}-(1_{A}-(p-g))a_{i}\|+\|(p-g)a_{i}-a'_{i}\|$
    	
    $+\|(p-g)^{\frac{1}{2}}a_{i}(p-g)^{\frac{1}{2}}-(p-g)a_{i}\|$

    	 $+\|(1_{A}-(p-g))a_{i}-(1_{A}-(p-g))^{\frac{1}{2}}a_{i}(1_{A}-(p-g))^{\frac{1}{2}}\|$
    	
    	$<\varepsilon+\varepsilon+\varepsilon+\varepsilon=4\varepsilon$ for $1\leq i\leq k$.

    	Since $\|(p-g)^{\frac{1}{2}}b_{2}(p-g)^{\frac{1}{2}}-b_{2}'\|<\varepsilon$, by (1) of Theorem \ref{thm:2.1}, we have $$(b_{2}'-3\varepsilon)_{+}\precsim((p-g)^{\frac{1}{2}}b_{2}(p-g)^{\frac{1}{2}}-2\varepsilon)_{+}. ~~~ \hspace{0.4cm}\textbf{(3.4.1)}$$
    	By $(4)'$, one has  $$\|(p-g)^{\frac{1}{2}}b_{2}(p-g)^{\frac{1}{2}}\|
    	\geq\|(p-g)b_{2}(p-g)\|\geq1-\varepsilon.$$
    	
    	Therefore, we have $\|(b_{2}'-3\varepsilon)_{+}\|\geq\|(p-g)^{\frac{1}{2}}b_{2}(p-g)^{\frac{1}{2}}\|+4\varepsilon\geq1-5\varepsilon$,
    	then, with $0<\varepsilon'<\varepsilon<\frac{1}{5}$, $(b_{2}'-3\varepsilon)_{+}\neq0$.
    	
    	Define a  completely positive contractive map $\varphi'':A\rightarrow A$ by $\varphi''(a)=(1_{A}-(p-g))^{\frac{1}{2}}a(1_{A}-(p-g))^{\frac{1}{2}}$.
    	Since $B$ is a nuclear $\rm C^{*}$-algebra, by Theorem 2.3.13 of \cite{L2}, there exists a contractive completely positive map $\psi'':A\rightarrow B$
    	such that $\|\psi''(p-g)-(p-g)\|<\varepsilon$,  and $\|\psi''(a'_{i})-a'_{i}\|<\varepsilon$ for all $1\leq i\leq n$.
    	
    	Since $B\in\Omega$, then $B$ has the new  type of   tracial nuclear dimension at most $n$, there exist a finite
    	dimension $\rm C^{*}$-algebra $F=F_{0}\oplus\cdots\oplus F_{n}$ and completely positive maps $\psi':B\rightarrow F$,
    	$\varphi':F\rightarrow B$ such that
    	
    	$(1)''$  for any $a_{i}'$ ($1\leq i\leq k$), there exists $\overline{\overline{a_{i}'}}\in B_+$, such that $\|a_{i}'-\overline{\overline{a_{i}'}}-\varphi'\psi'(a_{i}')\|<\varepsilon$, and for  $g\in B_+$, there exists $\overline{\overline{g}}\in B$, such that $\|g-\overline{\overline{g}}-\varphi'\psi'(g)\|<\varepsilon. ~~~~~~\hspace{0.4cm}\textbf{(3.4.2)}$
    	
    	$(2)''$ $(p-\varphi'\psi'(p)-\varepsilon)_+\precsim (b_{2}'-3\varepsilon)_{+}$,
    	
    	$(3)''$ $\|\psi'\|\leq1$, and
    	
    	$(4)''$ $\varphi'|_{F_{i}}$ is  a  completely positive contractive order zero map for $i=1, \cdots, n$.
    	
    	Define $\varphi:F\rightarrow A$ by $\varphi(a)=\varphi'(a)$ and $\psi:A\rightarrow F$ by
    	$\psi(a)=\psi'\psi''((p-g)^{\frac{1}{2}}a(p-g)^{\frac{1}{2}})$ and
    	$\overline{a_{i}}=\varphi''(a_{i})+\overline{\overline{a_{i}'}}\in A_+$ for $1\leq i\leq k$.
    	
    	Then, one has
    	
    	$(1_{A}-\varphi\psi(1_{A})-4\varepsilon)_+$

    $= (1_{A}-\varphi'\psi'\psi''(p-g)-4\varepsilon)_+$

    $\precsim 1_{A}-(\varphi'\psi'(p)-2\varepsilon)_++\varphi'\psi'(g)$

        $\precsim (1_{A}-p)+(p-\varphi'\psi'(p)-\varepsilon)_++\varphi'\psi'(g)$

    	$\precsim (1_{A}-p)+(p-\varphi'\psi'(p)-\varepsilon)_++\varphi'\psi'(g)+(\overline{\overline{g}}-\varepsilon)_+$
    	
    	$\precsim (1_{A}-p)+ g+(p-\varphi'\psi'(p)-\varepsilon)_+$ ~ (by \textbf{(3.4.2)})
    	
    	$\precsim b_{1}\oplus(b_{2}'-3\varepsilon)_{+}$ ~(by $(1)'$)
    	
        $\precsim b_{1}\oplus((p-g)^{\frac{1}{2}}b_{2}(p-g)^{\frac{1}{2}}-2\varepsilon)_{+}$~ (by \textbf{(3.4.1)})

        $\precsim b_{1}+b_{2}\precsim b$.
    	
    	One  also has
    	
    	$\|a_{i}-\overline{a_{i}}-\varphi\psi(a_{i})\|=\|a_{i}-\varphi''(a_{i})-
    	 \overline{\overline{a_{i}'}}-\varphi'\psi'\psi''((p-g)^{\frac{1}{2}}a_{i}(p-g)^{\frac{1}{2}})\|$
    	
    	$\leq\|a_{i}-(1_{A}-(p-g))^{\frac{1}{2}}a_{i}(1_{A}-(p-g))^{\frac{1}{2}}-
    	 \overline{\overline{a_{i}'}}-\varphi'\psi'\psi''((p-g)^{\frac{1}{2}}a_{i}(p-g)^{\frac{1}{2}})\|$
    	
    	 $\leq\|a_{i}-a'_{i}-(1_{A}-(p-g))^{\frac{1}{2}}a_{i}(1_{A}-(p-g))^{\frac{1}{2}}\|+
    	 \|a'_{i}-\overline{\overline{a_{i}'}}-\varphi'\psi'\psi''((p-g)^{\frac{1}{2}}a_{i}(p-g)^{\frac{1}{2}})\|$
    	    	$\leq3\varepsilon+
      	\|a'_{i}-\overline{\overline{a_{i}'}}-\varphi'\psi'(a_{i}')\|$
    	
    	$+\|\varphi'\psi'(a_{i}')-
    	\varphi'\psi'\psi''(a_{i}')\|+\| \varphi'\psi'\psi''(a_{i}')-\varphi'\psi'\psi''((p-g)^{\frac{1}{2}}a_{i}(p-g)^{\frac{1}{2}})\|$
    	    	$<3\varepsilon+2\varepsilon+2\varepsilon+2\varepsilon=9\varepsilon$.
    	
    	Since $\varphi'', \varphi', \psi', \psi''$ are  completely positive contractive maps, then $\varphi$ and $\psi$ are completely positive maps.
    	
    	For $(3)''$, $\varphi'|_{F_{i}}$ is a  completely positive contractive order zero map for $i=0, 1, \cdots, n$ and $\varphi(a)=\varphi'(a)$, then
    	$\varphi|_{F_{i}}$ is contractive completely positive order zero map for $i=0, 1,\cdots, n$.
    	
    For any $x\in A$,	 $\|\psi(x)\|=\|\psi'\psi''((p-g)^{\frac{1}{2}}x(p-g)^{\frac{1}{2}})\|\leq\|\psi'\|\|\psi''\|\|x\|$, then $\|\psi\|\leq\|\psi''\|\|\psi'\|\leq1$.
    	
    	Therefore, $A$ has the new  type of   tracial nuclear dimension at most $n$.
    \end{proof}

\begin{corollary}\label{cor:3.5} Let $A$ be a  unital simple  ${\rm C^*}$-algebra. Let $B\subseteq A$ be a centrally large nuclear subalgebra of $A$ such that $B$ has the new type of   tracial nuclear dimension at most $n$. Then   $A$ has the new type of  tracial nuclear dimension at most $n$.\end{corollary}

\begin{corollary}\label{cor:3.6} Let $\Omega$ be a class of   unital nuclear
${\rm C^*}$-algebras such that for  any $B\in \Omega$,  $B$ has the new type of  tracial nuclear dimension at most $n$. Then $A$ has the new type of  tracial nuclear dimension at most $n$ for  any  simple  unital ${\rm C^*}$-algebra $A\in \rm{TA}\Omega$.\end{corollary}

 \begin{theorem}\label{thm:3.7}
   Let $\Omega$ be a class of unital $\rm C^{*}$-algebras  such that $B$ is  weakly $(m, n)$-divisible (with $n\neq m$) (see Definition \ref{def:2.2}) for any $B\in \Omega$. Let $A\in {\rm  WTA}\Omega$
    	be a simple unital stably finite $\rm C^{*}$-algebra such that for any integer $n\in \mathbb{N}$ the $\rm C^{*}$-algebra ${\rm M}_{n}(A)$	 belongs to the class ${\rm WTA}\Omega$. Then $A$ is secondly weakly $(m, n)$-divisible (with $n\neq m$) (see Definition \ref{def:2.3}).
    \end{theorem}  \begin{proof}
   For  given $a\in {\rm M}_{\infty}(A)_{+},~\varepsilon>0$, we may assume that $a\in A_+$, and
    $\|a\|=1$, (we have replaced ${\rm M}_{n}(A)$ containing $a$ given initially by $A$), we must show that there are  $x_{1}, x_{2}, \cdots, x_{n}\in {\rm M}_{\infty}(A)_{+}$ such that
    $m\langle x_{j}\rangle \leq \langle a \rangle +\langle a \rangle$,
   and $\langle (a-\varepsilon)_{+} \rangle\leq\langle x_{1} \rangle+\langle x_{2} \rangle+\cdots+\langle x_{n} \rangle$ for all $j=1, 2, \cdots, n$.

    For any $\delta_{1}>0$, since $A\in {\rm WTA}\Omega$, there exist a projection $p\in A$, an element $g\in A$ with $0\leq g \leq1$,
    and a $\rm C^{*}$-subalgebra $B$ of $A$ with $g\in B$, $1_{B}=p$, and $B\in \Omega$ such that

    (1) $(p-g)a\in_{\delta_{1}}B$, and

    (2) $\|(p-g)a-a(p-g)\|<\delta_{1}$.

    By (2), with sufficiently small $\delta_{1}$, by Lemma 2.5.11 (1) of \cite{L2}, we have

    (3) $\|(p-g)^{\frac{1}{2}}a-a(p-g)^{\frac{1}{2}}\|<\varepsilon/3$, and

    (4) $\|(1-(p-g))^{\frac{1}{2}}a-a(1-(p-g))^{\frac{1}{2}}\|<\varepsilon/3$.

    By (1) and (2), with sufficiently small $\delta_{1}$, there exists a positive element $a^{'}\in B$ such that

    (5) $\|(p-g)^{\frac{1}{2}}a(p-g)^{\frac{1}{2}}-a^{'}\|<\varepsilon/3$.

    By (3), (4) and (5),

    $\|a-a^{'}-(1-(p-g))^{\frac{1}{2}}a(1-(p-g))^{\frac{1}{2}}\|$

    $\leq\|a-(p-g)a-(1-(p-g))a\|+\|(p-g)a-(p-g)^{\frac{1}{2}}a(p-g)^{\frac{1}{2}}\|$

    $+\|(1-(p-g))a-(1-(p-g))^{\frac{1}{2}}a(1-(p-g))^{\frac{1}{2}}\|$

    $+\|(p-g)^{\frac{1}{2}}a(p-g)^{\frac{1}{2}}-a^{'}\|$

    $<\varepsilon/3+\varepsilon/3+\varepsilon/3=\varepsilon$.

    Since $B$ is weakly $(m, n)$-divisible, and $(a^{'}-2\varepsilon)_{+}\in B_+$, there exist
    $x_{1}^{'}, x_{2}^{'},$ $\cdots, x_{n}^{'}\in B$, such that
    $\langle x_{j}^{'}\rangle +\langle x_{j}^{'}\rangle +\cdots+\langle x_{j}^{'}\rangle \leq\langle(a^{'}-2\varepsilon)_{+} \rangle $,
    and $\langle(a^{'}-3\varepsilon)_{+} \rangle\leq\sum_{i=1}^{n}\langle x_{i}^{'}\rangle$, where $\langle x_{j}^{'}\rangle $ repeats $m$ times.

    Since $B$ is weakly $(m, n)$-divisible, and $(a^{'}-\varepsilon)_{+}\in B$, there exist
    $y_{1}^{'}, y_{2}^{'},$ $\cdots, y_{n}^{'}\in B$, such that
    $\langle y_{j}^{'}\rangle +\langle y_{j}^{'}\rangle +\cdots+\langle y_{j}^{'}\rangle \leq\langle(a^{'}-\varepsilon)_{+} \rangle $,
    and $\langle(a^{'}-2\varepsilon)_{+} \rangle\leq\sum_{i=1}^{n}\langle y_{i}^{'}\rangle$, where $\langle y_{j}^{'}\rangle $ repeats $m$ times.

    Write $a^{''}=(1-(p-g))^{\frac{1}{2}}a(1-(p-g))^{\frac{1}{2}}$.

    We divide the proof into two cases.

    {\bf Case (1)} We assume that $(a^{'}-2\varepsilon)_{+}$ is Cuntz equivalent to a projection.

    {\bf (1.1)} We assume that $(a^{'}-3\varepsilon)_{+}$ is Cuntz equivalent to a projection.

    {\bf (1.1.1)} We assume that $(a^{'}-2\varepsilon)_{+}$ is Cuntz equivalent to  $(a^{'}-3\varepsilon)_{+}$.

    {\bf (1.1.1.1)} If  $x_{1}^{'}, x_{2}^{'}, \cdots, x_{n}^{'}\in B$ are all Cuntz equivalent to projections,
    and $\langle(a^{'}-3\varepsilon)_{+} \rangle\leq\sum_{i=1}^{n}\langle x_{i}^{'}\rangle$,
    then, by Theorem 2.1 (2), there exist some  integer $j$ and a nonzero projection $d$ such that $\langle x_{j}^{'}+d\rangle +\langle x_{j}^{'}+d\rangle+\cdots+\langle x_{j}^{'}+d\rangle \leq\langle(a^{'}-2\varepsilon)_{+} \rangle $, where $\langle x_{j}^{'}+d\rangle$ repeats $m$ times, otherwise, this contradicts the stably
    finiteness of $A$ (since $m\neq n$ and $\rm C^{*}$-algebra $A$ is stably finite).

    For any $\delta_{2}>0$, since $A\in {\rm WTA}\Omega$, there exist a projection $p_1\in A$, an element $g_{1}\in A$ with
    $0\leq g_{1} \leq1$, and a $\rm C^{*}$-subalgebra $D_1$ of $A$ with $g_{1}\in D_1$, $1_{D{_1}}=p_1$, and $D_1\in \Omega$ such that

    ($1^{'}$) $(p_1-g_{1})a^{''}\in_{\delta_{2}}D_1$,

    ($2^{'}$) $\|(p_1-g_{1})a^{''}-a^{''}(p_1-g_{1})\|<\delta_{2}$, and

    ($3^{'}$) $1-(p_1-g_{1})\preceq d$.

    By ($1^{'}$) and ($2^{'}$), with sufficiently small $\delta_{2}$, as above, via the analogues for (4), (5) and (6) for
    $a^{''},~ p_1,$ and $g_{1}$, there exist a positive element $a^{'''}\in D_1$, such that

    	$ \|(p_1-g_{1})^{\frac{1}{2}}a^{''}(p_1-g_{1})^{\frac{1}{2}}-a^{'''}\|<\varepsilon/3,$ and
   	$ \|a^{''}-a^{'''}-(1-(p_1-g_{1}))^{\frac{1}{2}}a^{''}(1-(p_1-g_{1}))^{\frac{1}{2}}\|<\varepsilon.$

    Since $D_1$ is weakly $(m, n)$-divisible, and  $(a^{'''}-2\varepsilon)_{+}\in D_1$, there exist positive elements
    $x_{1}^{''}, x_{2}^{''}, \cdots, x_{n}^{''}\in D_1$, such that

    $\langle x_{j}^{''}\rangle +\langle x_{j}^{''}\rangle +\cdots+\langle x_{j}^{''}\rangle \leq\langle(a^{'''}-2\varepsilon)_{+} \rangle $,
    and
  $\langle(a^{'''}-3\varepsilon)_{+} \rangle\leq\sum_{i=1}^{n}\langle x_{i}^{''}\rangle$, where $\langle x_{j}^{''}\rangle $ repeats $m$ times.

    Since $a^{'}\leq a^{'}+a^{''},$ then $\langle (a^{'}-\varepsilon)_{+}\rangle\leq \langle (a^{'}+a^{''}-\varepsilon)_{+}\rangle$, and $\|a-a^{'}-a^{''}\|<\varepsilon$, so  one has $\langle (a^{'}+a^{''}-\varepsilon)_{+}\rangle\leq\langle a\rangle$.

    Therefore, one has
    $$\langle (a^{'}-2\varepsilon)_{+}\rangle\leq\langle (a^{'}-\varepsilon)_{+}\rangle\leq\langle a\rangle.$$

    Write $x=(1-(p_1-g_{1}))^{\frac{1}{2}}a^{''}(1-(p_1-g_{1}))^{\frac{1}{2}}$.

    Since
    $a^{'''}\leq a^{'''}+x$, then $\langle (a^{'''}-2\varepsilon)_{+}\rangle\leq\langle (a^{'''}+x-\varepsilon)_{+}\rangle$, and
    $\|a^{''}-a^{'''}-x\|<\varepsilon$, which implies that
    $$\langle (a^{'''}-2\varepsilon)_{+}\rangle\leq\langle a^{''}\rangle \leq\langle a\rangle.$$

    Therefore, we have

    	$\langle (x_{j}^{'}\oplus d)\oplus x_{j}^{''}\rangle +\langle (x_{j}^{'}\oplus d)\oplus x_{j}^{''}\rangle +\cdots+\langle (x_{j}^{'}\oplus d)+ x_{j}^{''}\rangle$

    	$ \leq\langle (a^{'}-2\varepsilon)_{+}\rangle+\langle (a^{'''}-2\varepsilon)_{+}\rangle\leq
    	\langle a\rangle+\langle a\rangle$,
    where $\langle (x_{j}^{'}\oplus d)\oplus x_{j}^{''}\rangle $ repeats $m$ times, and

    	$\langle x_{i}^{'}\oplus x_{i}^{''}\rangle +\langle x_{i}^{'}\oplus x_{i}^{''}\rangle +\cdots\langle x_{i}^{'}\oplus x_{i}^{''}\rangle$,

    	$\leq\langle (a^{'}-2\varepsilon)_{+}\rangle+\langle (a^{'''}-2\varepsilon)_{+}\rangle\leq
    	\langle a\rangle+\langle a\rangle$,
    where $\langle x_{i}^{'}\oplus x_{i}^{''}\rangle $ repeats $m$ times for $1\leq i\leq n$ and $i\neq j$.

    We also have

    $\langle (a-10\varepsilon)_{+}\rangle$

    $\leq\langle (a^{'}-3\varepsilon)_{+}\rangle+\langle (a^{'''}-3\varepsilon)_{+}\rangle+\langle (1-(p_1-g_{1}))^{\frac{1}{2}}a^{''}(1-(p_1-g_{1}))^{\frac{1}{2}}\rangle$

    $\leq\langle (a^{'}-3\varepsilon)_{+}\rangle+\langle (a^{'''}-3\varepsilon)_{+}\rangle+\langle (1-(p_1-g_{1}))\rangle$

    $\leq\langle (a^{'}-3\varepsilon)_{+}\rangle+\langle (a^{'''}-3\varepsilon)_{+}\rangle+\langle d\rangle$

    $\leq\sum_{i=1, i\neq j}
    ^n\langle x_{i}^{'}\oplus x_{i}^{''}\rangle+\langle (x_{j}^{'}\oplus d)\oplus x_{j}^{''}\rangle.$

    There are the desired inequalities, with$\langle x_{i}^{'}\oplus x_{i}^{''}\rangle+\langle (x_{j}^{'}\oplus d)\oplus x_{j}^{''}\rangle$ in place of $\langle  x_{i}\rangle$, and $10\varepsilon$ in place of $\varepsilon$.

    {\bf (1.1.1.2)} If  $x_{1}^{'}, x_{2}^{'}, \cdots, x_{k}^{'}\in B$ are  projections, and
    $\langle(a^{'}-3\varepsilon)_{+} \rangle <\sum_{i=1}^{n}\langle x_{i}^{'}\rangle$,
    then, by Theorem 2.1 (2), there exists a nonzero projection $e$, such that $\langle(a^{'}-3\varepsilon)_{+} \rangle+\langle e\rangle \leq\sum_{i=1}^{n}\langle x_{i}^{'}\rangle$.

    As in the part (1.1.1.1), since  $A\in {\rm  WTA}\Omega$, there exist a projection $p_2\in A$, an element $g_{2}\in A$ with
    $0\leq g_{2} \leq1$, and a $\rm C^{*}$-subalgebra $D_2$ of $A$ with $g_{2}\in D_2$, $1_{D_{2}}=p_2$, and $D_2\in \Omega$, by $(1)'$,
    there exists a positive element $a^{4}\in D_2$, such that

    $1-(p_2-g_{2})\preceq e,$
    	 $\|(p_2-g_{2})^{\frac{1}{2}}a^{''}(p_2-g_{2})^{\frac{1}{2}}-a^{4}\|<\varepsilon/3,$ and

    	 $\|a^{''}-a^{4}-(1-(p_2-g_{2}))^{\frac{1}{2}}a^{''}(1-(p_2-g_{2}))^{\frac{1}{2}}\|<\varepsilon.$

Also as in the part (1.1.1.1),
we have $\langle (a^{4}-2\varepsilon)_{+}\rangle\leq \langle a\rangle.$

    Since $D_2$ is weakly $(m, n)$-divisible, $(a^{4}-2\varepsilon)_{+}\in D_2$, there exist
    $x_{1}^{4}, x_{2}^{4},$ $\cdots, x_{n}^{4}\in D_2$, such that

    $\langle x_{j}^{4}\rangle +\langle x_{j}^{4}\rangle +\cdots+\langle x_{j}^{4}\rangle \leq\langle(a^{4}-2\varepsilon)_{+} \rangle $,
    and

     $\langle(a^{4}-3\varepsilon)_{+} \rangle\leq\sum_{i=1}^{n}\langle x_{i}^{4}\rangle$, where $\langle x_{j}^{4}\rangle $ repeats $m$ times.

    Therefore, we have

    	$\langle x_{j}^{'}\oplus x_{j}^{4}\rangle +\langle x_{j}^{'}\oplus x_{j}^{4}\rangle +\cdots+\langle x_{j}^{'}\oplus x_{j}^{4}\rangle$

    	$\leq\langle (a^{'}-2\varepsilon)_{+}\rangle+\langle (a^{4}-2\varepsilon)_{+}\rangle\leq
    	\langle a\rangle+\langle a\rangle$,
    where $\langle x_{j}^{'}\oplus x_{j}^{4}\rangle $ repeats $m$ time for  $1\leq i\leq n$.

    We also have

    $\langle (a-10\varepsilon)_{+}\rangle$

    $\leq\langle (a^{'}-3\varepsilon)_{+}\rangle+\langle (a^{4}-3\varepsilon)_{+}\rangle+\langle (1-(p_{2}-g_{2}))^{\frac{1}{2}}a^{''}(1-(p_{2}-g_{2}))^{\frac{1}{2}}\rangle$

    $\leq\langle (a^{'}-3\varepsilon)_{+}\rangle+\langle (a^{4}-3\varepsilon)_{+}\rangle+\langle (1-(p_{2}-g_{2}))\rangle$

    $\leq\langle (a^{'}-3\varepsilon)_{+}\rangle+\langle (a^{4}-3\varepsilon)_{+}\rangle+\langle e\rangle$

    $ \leq\sum_{i=1}^{n}\langle x_{i}^{'}\oplus x_{i}^{4}\rangle.$

    There are the desired inequalities, with $\langle x_{i}^{'}\oplus x_{i}^{4}\rangle$ in place of $\langle  x_{i}\rangle$, and $10\varepsilon$ in place of $\varepsilon$.

    {\bf (1.1.1.3)} we assume that there is a purely positive element, and we assume $x_{1}^{'}$ is a purely positive element. As
    $\langle(a^{'}-2\varepsilon)_{+} \rangle \leq\sum_{i=1}^{n}\langle x_{i}^{'}\rangle$,
    for any $\varepsilon>0$, there exists $\delta>0$ such that $\langle(a^{'}-4\varepsilon)_{+} \rangle\leq \langle (x_{1}^{'}-\delta)_{+}\rangle \oplus\sum_{i=1}^{n}\langle x_{i}^{'}\rangle$. Since $x_{1}^{'}$ is a purely positive element,  by Theorem 2.1 (3), there exists a nonzero positive element $s$ such that
    $\langle(x_{1}^{'}-\delta)_{+} \rangle +\langle s\rangle\leq\langle x_{1}^{'}\rangle$.

    As in the part (1.1.1.1), since $A\in {\rm WTA}\Omega$, by $(1)'$, there exists a projection $p_3\in A$, an element $g_{3}\in A$ with
    $0\leq g_{3} \leq1$, and a $\rm C^{*}$-subalgebra $D_3$ of $A$ with $g_{3}\in D_3$, $1_{D_{3}}=p_3$, and $D_3\in \Omega$
    there exist a positive element $a^{5}\in D_3$, such that

        $1-(p_3-g_{3})\preceq s,$
    	 $\|(p_3-g_{3})^{\frac{1}{2}}a^{''}(p_3-g_{3})^{\frac{1}{2}}-a^{5}\|<\varepsilon/3,$ and

    	 $\|a^{''}-a^{5}-(1-(p_3-g_{3}))^{\frac{1}{2}}a^{''}(1-(p_3-g_{3}))^{\frac{1}{2}}\|<\varepsilon.$

Also as in the part (1.1.1.1),
we have $\langle (a^{5}-2\varepsilon)_{+}\rangle\leq \langle a\rangle.$

    Since $D_3$ is weakly $(m, n)$-divisible, $(a^{5}-2\varepsilon)_{+}\in D_3$, there exist
    $x_{1}^{5}, x_{2}^{5},$ $\cdots, x_{n}^{5}\in D_3$, such that

    $\langle x_{j}^{5}\rangle +\langle x_{j}^{5}\rangle +\cdots+\langle x_{j}^{5}\rangle \leq\langle(a^{5}-2\varepsilon)_{+} \rangle $,
    and

$\langle(a^{5}-3\varepsilon)_{+} \rangle\leq\sum_{i=1}^{n}\langle x_{i}^{5}\rangle$, where $\langle x_{j}^{5}\rangle $ repeats $m$ times.

    Therefore, we have

    	$\langle x_{j}^{'}\oplus x_{j}^{5}\rangle +\langle x_{j}^{'}\oplus x_{j}^{5}\rangle +\cdots+\langle x_{j}^{'}\oplus x_{j}^{5}\rangle$,

    	$\leq\langle (a^{'}-2\varepsilon)_{+}\rangle+\langle (a^{5}-2\varepsilon)_{+}\rangle\leq
    	\langle a\rangle+\langle a\rangle$,
    where $\langle x_{i}^{'}\oplus x_{i}^{5}\rangle $ repeats $m$ times for $1\leq  i\leq n$.

   We also have

    $\langle (a-10\varepsilon)_{+}\rangle$

    $\leq\langle (a^{'}-4\varepsilon)_{+}\rangle+\langle (a^{5}-3\varepsilon)_{+}\rangle+\langle (1-(p_3-g_{3}))^{\frac{1}{2}}a^{''}(1-(p_3-g_{3}))^{\frac{1}{2}}\rangle$

    $\leq\langle (a^{'}-4\varepsilon)_{+}\rangle+\langle (a^{5}-3\varepsilon)_{+}\rangle+\langle (1-(p_3-g_{3}))\rangle$

    $\leq\langle (a^{'}-4\varepsilon)_{+}\rangle+\langle (a^{5}-3\varepsilon)_{+}\rangle+\langle s\rangle$
    $\leq\sum_{i=1}^{n}\langle x_{i}^{'}\oplus x_{i}^{5}\rangle.$

    There are the desired inequalities, with $\langle x_{i}^{'}\oplus x_{i}^{5}\rangle$ in place of $\langle  x_{i}\rangle$, and $10\varepsilon$ in place of $\varepsilon$.

    {\bf (1.1.2)} we assume that exists nonzero projection $r$ such that  $\langle (a^{'}-3\varepsilon)_{+}\rangle
    +\langle r\rangle \leq\langle  (a^{'}-3\varepsilon)_{+} \rangle$.

    As in the part (1.1.1.1), as $A\in {\rm WTA}\Omega$, there exist a projection $p_4\in A$, an element $g_{4}\in A$ with
    $0\leq g_{4} \leq1$, and a $\rm C^{*}$-subalgebra $D_4$ of $A$ with $g_{4}\in D_4$, $1_{D_{4}}=p_4$, and $D_4\in \Omega$, by $(1)'$,
    there exists a positive element $a^{6}\in D_4$, such that

    $1-(p_4-g_{4})\preceq r,$
    	 $\|(p_4-g_{4})^{\frac{1}{2}}a^{''}(p_4-g_{4})^{\frac{1}{2}}-a^{6}\|<\varepsilon/3,$ and

    	 $\|a^{''}-a^{6}-(1-(p_4-g_{4}))^{\frac{1}{2}}a^{''}(1-(p_4-g_{4}))^{\frac{1}{2}}\|<\varepsilon.$

Also as in the part (1.1.1.1),
we have $\langle (a^{6}-2\varepsilon)_{+}\rangle\leq \langle a\rangle.$

    Since $D_4$ is weakly $(m, n)$-divisible, $(a^{6}-2\varepsilon)_{+}\in D_4$, there exist
    $x_{1}^{6}, x_{2}^{6},$ $\cdots, x_{n}^{6}\in D_4$, such that

    $\langle x_{j}^{6}\rangle +\langle x_{j}^{6}\rangle +\cdots+\langle x_{j}^{6}\rangle \leq\langle(a^{6}-2\varepsilon)_{+} \rangle $,
    and

  $\langle(a^{6}-3\varepsilon)_{+} \rangle\leq\sum_{i=1}^{n}\langle x_{i}^{6}\rangle$, where $\langle x_{j}^{6}\rangle $ repeats $m$ times.

    Therefore, we have

    	$\langle y_{i}^{'}\oplus x_{i}^{6}\rangle +\langle y_{i}^{'}\oplus x_{i}^{6}\rangle +\cdots+\langle y_{i}^{'}\oplus x_{i}^{6}\rangle$

    	$\leq\langle (a^{'}-\varepsilon)_{+}\rangle+\langle (a^{6}-2\varepsilon)_{+}\rangle\leq
    	\langle a\rangle+\langle a\rangle$,
    where $\langle y_{i}^{'}\oplus x_{i}^{6}\rangle $ repeats $m$ times for $1\leq i\leq n$.

    We also have

    $\langle (a-10\varepsilon)_{+}\rangle$

    $\leq\langle (a^{'}-3\varepsilon)_{+}\rangle+\langle (a^{6}-4\varepsilon)_{+}\rangle+\langle (1-(p_4-g_{4}))^{\frac{1}{2}}a^{''}(1-(p_4-g_{4}))^{\frac{1}{2}}\rangle$

    $\leq\langle (a^{'}-3\varepsilon)_{+}\rangle+\langle (a^{6}-4\varepsilon)_{+}\rangle+\langle (1-(p_4-g_{4}))\rangle$

    $\leq\langle (a^{'}-3\varepsilon)_{+}\rangle+\langle (a^{6}-4\varepsilon)_{+}\rangle+\langle r\rangle$
    $\leq\sum_{i=1}^{n}\langle y_{i}^{'}\oplus x_{i}^{6}\rangle.$

    There are the desired inequalities, with$\langle y_{i}^{'}\oplus x_{i}^{6}\rangle$ in place of $\langle  x_{i}\rangle$, and $10\varepsilon$ in place of $\varepsilon$.

    {\bf (1.2)} If $(a^{'}-3\varepsilon)_{+}$ is not Cuntz equivalent to a projection,  by Theorem 2.1 (3), then there is a nonzero positive
    element $d$ such that $\langle (a^{'}-4\varepsilon)_{+}\rangle+\langle d\rangle \leq\langle (a^{'}-3\varepsilon)_{+}\rangle$.

    As in the part (1.1.1.1), since  $A\in {\rm WTA}\Omega$, there exist a projection $p_5\in A$, an element $g_{5}\in A$ with
    $0\leq g_{5} \leq1$, and a $\rm C^{*}$-subalgebra $D_5$ of $A$ with $g_{5}\in D_5$, $1_{D_{5}}=p_5$, and $D_5\in \Omega$, by $(1)'$,
    there exists a positive element $a^{7}\in D_5$, such that

    $1-(p_5-g_{5})\preceq d,$
    	 $\|(p_5-g_{5})^{\frac{1}{2}}a^{''}(p_5-g_{5})^{\frac{1}{2}}-a^{7}\|<\varepsilon/3,$ and

    	 $\|a^{''}-a^{7}-(1-(p_5-g_{5}))^{\frac{1}{2}}a^{''}(1-(p_5-g_{5}))^{\frac{1}{2}}\|<\varepsilon.$

Also as in the part (1.1.1.1),
we have $\langle (a^{7}-2\varepsilon)_{+}\rangle\leq \langle a\rangle.$

    Since $D_5$ is weakly $(m, n)$-divisible, $(a^{7}-2\varepsilon)_{+}\in D_5$, there exist
    $x_{1}^{7}, x_{2}^{7},$ $\cdots, x_{n}^{7}\in D_5$, such that

    $\langle x_{j}^{7}\rangle +\langle x_{j}^{7}\rangle +\cdots+\langle x_{j}^{7}\rangle \leq\langle(a^{7}-2\varepsilon)_{+}\rangle$,
    and
    $\langle(a^{7}-3\varepsilon)_{+} \rangle\leq\sum_{i=1}^{n}\langle x_{i}^{7}\rangle$, where $\langle x_{j}^{7}\rangle$ repeats $m$ times.

    Therefore, we have

    	$\langle x_{j}^{'}\oplus x_{j}^{7}\rangle +\langle x_{j}^{'}\oplus x_{j}^{7}\rangle +\cdots+\langle x_{j}^{'}\oplus x_{j}^{7}\rangle$

    	$\leq\langle (a^{'}-2\varepsilon)_{+}\rangle+\langle (a^{7}-2\varepsilon)_{+}\rangle\leq
    	\langle a\rangle+\langle a\rangle$,
    where $\langle x_{j}^{'}\oplus x_{j}^{7}\rangle $ repeats $m$ times for $1\leq i\leq n$.

    We also have

    $\langle (a-10\varepsilon)_{+}\rangle$

    $\leq\langle (a^{'}-4\varepsilon)_{+}\rangle+\langle (a^{7}-4\varepsilon)_{+}\rangle+\langle (1-(p_5-g_{5}))^{\frac{1}{2}}a^{''}(1-(p_5-g_{5}))^{\frac{1}{2}}\rangle$

    $\leq\langle (a^{'}-4\varepsilon)_{+}\rangle+\langle (a^{7}-4\varepsilon)_{+}\rangle+\langle (1-(p_5-g_{5}))\rangle$

    $\leq\langle (a^{'}-4\varepsilon)_{+}\rangle+\langle (a^{7}-4\varepsilon)_{+}\rangle+\langle d\rangle$

    $\leq\sum_{i=1}^{n}\langle x_{i}^{'}\oplus x_{i}^{7}\rangle.$

    There are the desired inequalities, with $\langle x_{i}^{'}\oplus x_{i}^{7}\rangle$ in place of $\langle  x_{i}\rangle$, and $10\varepsilon$ in place of $\varepsilon$.

    {\bf Case(2)} If $(a^{'}-2\varepsilon)_{+}$ is not Cuntz equivalent to a projection, by Theorem 2.1 (3), there is a nonzero positive
    element $d$ such that $\langle (a^{'}-3\varepsilon)_{+}\rangle+\langle d\rangle \leq\langle (a^{'}-2\varepsilon)_{+}\rangle$.

    As in the part (1.1.1.1), since  $A\in {\rm WTA}\Omega$, there exist a projection $p_6\in A$, an element $g_{6}\in A$ with
    $0\leq g_{6} \leq1$, and a $\rm C^{*}$-subalgebra $D_6$ of $A$ with $g_{6}\in D_6$, $1_{D_{6}}=p_6$, and $D_6\in \Omega$, by $(1)'$,
    there exists a positive element $a^{8}\in D_6$, such that

    $1-(p_6-g_{6})\preceq d,$
    	 $\|(p_6-g_{6})^{\frac{1}{2}}a^{''}(p_6-g_{6})^{\frac{1}{2}}-a^{8}\|<\varepsilon/3,$ and

    	 $\|a^{''}-a^{8}-(1-(p_6-g_{6}))^{\frac{1}{2}}a^{''}(1-(p_6-g_{6}))^{\frac{1}{2}}\|<\varepsilon.$

Also as in the part (1.1.1.1),
we have $\langle (a^{8}-2\varepsilon)_{+}\rangle\leq \langle a\rangle.$

    Since $D_6$ is weakly $(m, n)$-divisible, $(a^{8}-2\varepsilon)_{+}\in D_6$, there exist
    $x_{1}^{8}, x_{2}^{8},$ $\cdots, x_{n}^{8}\in D_6$, such that

    $\langle x_{j}^{8}\rangle +\langle x_{j}^{8}\rangle +\cdots+\langle x_{j}^{8}\rangle \leq\langle(a^{8}-2\varepsilon)_{+} \rangle$,
    and
   $\langle(a^{8}-3\varepsilon)_{+} \rangle\leq\sum_{i=1}^{n}\langle x_{i}^{8}\rangle$, where $\langle x_{j}^{8}\rangle $ repeats $m$ times.

    Therefore, we have

    	$\langle y_{i}^{'}\oplus x_{i}^{8}\rangle +\langle y_{i}^{'}\oplus x_{i}^{8}\rangle +\cdots+\langle y_{i}^{'}\oplus x_{i}^{8}\rangle$

    	$\leq\langle (a^{'}-\varepsilon)_{+}\rangle+\langle (a^{8}-2\varepsilon)_{+}\rangle\leq
    	\langle a\rangle+\langle a\rangle$,
    where $\langle y_{i}^{'}\oplus x_{i}^{8}\rangle $ repeats $m$ times for $1\leq i\leq n$.

    We also have

    $\langle (a-10\varepsilon)_{+}\rangle$

    $\leq\langle (a^{'}-3\varepsilon)_{+}\rangle+\langle (a^{8}-4\varepsilon)_{+}\rangle+\langle (1-(p_6-g_{6}))^{\frac{1}{2}}a^{''}(1-(p_6-g_{6}))^{\frac{1}{2}}\rangle$

    $\leq\langle (a^{'}-3\varepsilon)_{+}\rangle+\langle (a^{8}-4\varepsilon)_{+}\rangle+\langle (1-(p_6-g_{6}))\rangle$

    $\leq\langle (a^{'}-3\varepsilon)_{+}\rangle+\langle (a^{8}-4\varepsilon)_{+}\rangle+\langle d\rangle$
    $\leq\sum_{i=1}^{n}\langle y_{i}^{'}\oplus x_{i}^{8}\rangle.$

    There are the desired inequalities, with $\langle y_{i}^{'}\oplus x_{i}^{8}\rangle$ in place of $\langle  x_{i}\rangle$, and $10\varepsilon$ in place of $\varepsilon$.
    \end{proof}

The following corollary were obtained by Fan, Fang,  and
Zhao in \cite{FFZ}.

\begin{corollary}\label{cor:3.8} Let $A$ be a unital simple stably finite  separable  ${\rm C^*}$-algebra. Let $B\subseteq A$ be a centrally large subalgebra of  $A$ such that $B$ is weakly $(m, n)$-divisible. Then   $A$ is  secondly weakly $(m, n)$-divisible.\end{corollary}

 \end{document}